\documentclass{amsart}

\usepackage{amsmath, amsthm, amssymb, latexsym}

\newcommand{\diverges}{\mathord{\uparrow}}

\DeclareMathOperator{\dom}{dom}

\DeclareMathOperator{\gra}{graph}

\newcommand{\uhr}{\upharpoonright}

\renewcommand{\geq}{\geqslant}
\renewcommand{\leq}{\leqslant}

\newtheorem{thm}{Theorem}

\theoremstyle{definition}
\newtheorem{defn}[thm]{Definition}
\newtheorem{q}[thm]{Question}
\newtheorem{oq}[thm]{Open Question}

\title{A Minimal Pair in the Generic Degrees}

\author{Denis R. Hirschfeldt}

\address{Department of Mathematics, The University of Chicago}
\email{drh@math.uchicago.edu}

\thanks{Partially supported by grants DMS-1600543 and DMS-1854279 from
the National Science Foundation. I thank Carl Jockusch for reading and
commenting on a draft of the proof of the main result, and
later a draft of the paper itself.}

\begin{document}

\begin{abstract}
We show that there is a minimal pair in the nonuniform generic
degrees, and hence also in the uniform generic degrees. This fact
contrasts with Igusa's result that there are no minimal pairs for
relative generic computability, and answers a basic structural
question mentioned in several papers in the area.
\end{abstract}

\maketitle

Generic computability is a notion of ``almost everywhere
computability'' introduced by Kapovich, Myasnikov, Schupp, and
Shpilrain~\cite{KMSS}. Beginning with the work of Jockusch and
Schupp~\cite{JocSch}, it and the related notion of coarse
computability have been studied from the computability-theoretic
viewpoint by several authors. (See Jockusch and Schupp~\cite{JocSch2}
for a survey. Coarse computability had actually been consider earlier
by Terwijn~\cite{Terwijn}.) Here, ``almost everywhere'' is defined in
terms of (asymptotic) density. A set $A$ has \emph{density $1$} if
$\lim_n \frac{|A \uhr n|}{n} = 1$, and has \emph{density $0$} if
its complement has density $1$.

\begin{defn}
A \emph{generic description} of a set $A$ is a partial function $f$
such that $\dom f$ has density $1$ and $f(n)=A(n)$ whenever $f(n)$ is
defined. A set is \emph{generically computable} if it has a partial
computable generic description.

A \emph{coarse description} of a set $A$ is a set $C$ such that $\{n :
C(n)=A(n)\}$ has density $1$. A set is \emph{coarsely computable} if
it has a computable coarse description.
\end{defn}

Thus generic computability captures the idea of computing a set while
allowing for a small number of errors of omission, while coarse
computability captures the idea of computing a set while allowing for
a small number of errors of commission. We can also consider notions
that allow both kinds of errors, as was done by Astor, Hirschfeldt,
and Jockusch~\cite{AHJ}.

We can of course relativize the above notions to an oracle. We can
also use them to define notions of reducibility. For coarse
reducibility, doing so is straightforward, though there are two
natural versions. (The fact that these versions are different, as are
the analogous ones for generic computability defined below, was shown
by~Dzhafarov and Igusa~\cite{DI}.)

\begin{defn}
We say that $A$ is \emph{nonuniformly coarsely reducible} to $B$ if
every coarse description of $A$ computes a coarse description of $B$.
We say that $A$ is \emph{uniformly coarsely reducible} to $B$ if there
is a Turing functional $\Phi$ such that if $C$ is a coarse description
of $B$, then $\Phi^C$ is a coarse description of $A$.
\end{defn}

Generic descriptions are partial functions, so we cannot use them
directly as oracles, but we can use their graphs, together with the
notion of enumeration reducibility, which does not allow us to use
negative information about an oracle. Recall that an
\emph{enumeration operator} is a c.e.\ set $W$ of pairs $(F,k)$
with each $F$ finite. For an oracle $X$, let $W^X = \{k : \exists
(F,k) \in W\, [F \subseteq X]\}$. Then $Y$ is \emph {enumeration
reducible to} $X$ if there is an enumeration operator $W$ such that
$Y = W^X$. We identify a partial function with its graph, so
for partial functions $f$ and $g$, we say that $f$ is
enumeration reducible to $g$ if $\gra(f)$ is enumeration
reducible to $\gra(g)$, and write $W^g$ for $W^{\gra(g)}$. We write
$W^X[s]$ for $\{k : \exists (F,k) \in W[s]\, [F \subseteq X]\}$. The
\emph{use} of an enumeration $k \in W^X[s]$ is the least $n$ such that
$\exists (F,k) \in W[s]\, [F \subseteq X \uhr n]$.

\begin{defn}
We say that $A$ is \emph{nonuniformly generically reducible} to $B$ if
for every generic description $f$ of $B$, there is a generic
description of $A$ that is is enumeration reducible to $f$.
We say that $A$ is \emph{uniformly generically reducible} to $B$ if
there is an enumeration operator $W$ such that if $f$ is a generic
description of $B$, then $W^f$ is a generic description of $A$.
\end{defn}

These reducibilities induce equivalence relations on $2^\omega$, from
which degree structures arise as usual. For any degree structure with
a least element $\bf 0$, a \emph{minimal pair} is a pair of degrees
${\bf a},{\bf b} > {\bf 0}$ such that if ${\bf c}<{\bf a}$ and ${\bf
c} < {\bf b}$ then ${\bf c} = {\bf 0}$. Most degree structures
studied in computability theory have minimal pairs, and proving this
fact is often one of the first structural results one establishes
about such a structure. For the (nonuniform and uniform) coarse
degrees, the existence of minimal pairs was proved by Hirschfeldt,
Jockusch, Kuyper, and Schupp~\cite{HJKS}. For the generic degrees,
however, the following basic question had remained open.

\begin{q}[Jockusch and Schupp~\cite{JocSch}; Igusa~\cite{Igusa}; see
also~\cite{Hirschfeldt}]
\label{question}
Are there minimal pairs in the (nonuniform or uniform) generic
degrees?
\end{q}

An interesting aspect of this question is its relationship to relative
generic computability. In the case of coarse computability, the
aforementioned proof of the existence of minimal pairs follows from
the following stronger result. (See e.g.~\cite{DowHir} for more on
notions of algorithmic randomness.)

\begin{thm}[Hirschfeldt, Jockusch, Kuyper, and Schupp~\cite{HJKS}]
\label{randthm}
If $X$ is not coarsely computable and $Y$ is weakly $3$-random
relative to $X$ then $X$ and $Y$ form a minimal pair for relative
coarse computability. That is, if a set is coarsely computable
relative both to $X$ and to $Y$, then it is coarsely computable.
\end{thm}

In the generic case, however, the situation is quite different.

\begin{thm}[Igusa~\cite{Igusa}]
\label{igthm}
There are no minimal pairs for relative generic computability. That
is, if $X$ and $Y$ are not computable, then there is a set that is
not generically computable, but is generically computable relative
both to $X$ and to $Y$.
\end{thm}

Notice that this result does not imply that there are no minimal pairs
for generic reducibility, because being generically computable
relative both to $X$ and to $Y$ is a weaker condition than being
generically reducible both to $X$ and to $Y$. For a set $A$ to have
the latter property, not only do both $X$ and $Y$ have to enumerate
generic descriptions of $A$, but so does every generic description of
$X$ or of $Y$. Thus we have some extra power in trying to build a
minimal pair for generic reducibility, which we will exploit below,
relying particularly on the fact that enumeration reducibility cannot
make use of information about inputs on which a partial computable
function is undefined, and hence is monotonic, in the sense that for
an enumeration operator $W$, if the partial function $g$ extends the
partial function $f$, then any number enumerated into $W^f$ is also
enumerated into $W^g$.

On the other hand, Theorem~\ref{igthm} does show that the methods
of~\cite{HJKS}, and the related ones of~\cite{AHJ}, are not available
here. Nevertheless, in this paper we give a positive answer to
Question~\ref{question} in both cases. (Notice that it is enough to
consider the nonuniform case, since if the nonuniform generic degrees
of two sets form a minimal pair, then so do their uniform generic
degrees.) Despite solving a reasonably well-known open problem, the
proof is fairly short. It is inspired by the construction of a minimal
pair of c.e.\ Turing degrees, but is a finite injury construction that
relies on the monotonicity of enumeration operators.

We now give the proof, followed by a few comments on coarse
computability, further work of Igusa~\cite{Igusa2}, and the notions of
dense and effective dense computability studied in~\cite{AHJ}.

\begin{thm}
\label{main}
There is a minimal pair in the nonuniform generic degrees (and hence
in the uniform generic degrees).
\end{thm}

\begin{proof}
Let $R_e = \{n : 2^e \mid n \, \wedge \, 2^{e+1} \nmid n\}$. Let
$W_0,W_1,\ldots$ be an effective listing of the enumeration operators.

We will build $\Delta^0_2$ sets $A_0$ and $A_1$, and the following
functions defined from these sets: Let $f_{j,s}$ be the partial
function defined by letting $f_{j,s}(n)=1$ if $n \notin A_j[s]$, and
$f_{j,s}(n)\diverges$ if $n \in A_j[s]$. Let $f_j$ be the partial
function defined by letting $f_j(n)=1$ if $n \notin A_j$, and
$f_j(n)\diverges$ if $n \in A_j$.

We will build $A_0$ and $A_1$ to have the following properties.
\begin{enumerate}

\item Each $A_j \cap R_e$ is finite, so $A_j$ is almost entirely
contained in each $\bigcup_{e>k} R_e=\{n>0 : 2^{e+1} \mid n\}$, and
hence has density $0$.

\item If $\dom \Phi_e \cap R_e$ is infinite then $\dom \Phi_e \cap R_e
\cap A_j \neq \emptyset$.

\item For each $e_0$, $e_1$, and $s$, if $x \in W_{e_0}^{f_{0,s}}[s] \cap
W_{e_1}^{f_{1,s}}[s]$, then $x \in W_{e_j}^{f_j}$ for some $j \leq 1$.

\end{enumerate}
Assuming we have done so, if $n \in \dom \Phi_e \cap R_e \cap A_j$
then let $X_j(n) \neq \Phi_e(n)$, and let $X_j(n)=1$ for
all other $n$. We claim that the generic degrees of $X_0$ and $X_1$
form a minimal pair

If $\dom \Phi_e$ has density $1$ then $\dom \Phi_e \cap R_e$ is
infinite, so by property (2), there is an $n \in \dom \Phi_e \cap R_e
\cap A_j$, and hence $\Phi_e$ is not a generic description of
$X_j$. Thus neither $X_j$ is generically computable.

By property (1) and the definition of $X_j$, each $f_j$ is a generic
description of $X_j$, so if $Y$ is generically reducible to both
$X_0$ and $X_1$, then there are $e_0$ and $e_1$ such that each
$W_{e_j}^{f_j}$ is a generic description of $Y$. Let $\Psi$ be defined
as follows. For each $n$,
search for an $s$ and a $k \leq 1$ such that $\langle n,k \rangle
\in W_{e_0}^{f_{0,s}}[s] \cap W_{e_1}^{f_{1,s}}[s]$. If one is found,
then let $\Psi(n)=k$. By property (3), if $\Psi(n)=k$ then $\langle n,k
\rangle \in W_{e_j}^{f_j}$ for some $j \leq 1$, which implies that
$Y(n)=k$. There are density $1$ many $n$ such that $\langle n,Y(n)
\rangle \in W_{e_0}^{f_0} \cap W_{e_1}^{f_1}$. For any such $n$, there
must be an $s$ such that $\langle n,Y(n) \rangle \in
W_{e_0}^{f_{0,s}}[s] \cap W_{e_1}^{f_{1,s}}[s]$, whence
$\Psi(n)=Y(n)$. Thus $\Psi$ is a generic description of $Y$. But
$\Psi$ is partial computable, so $Y$ is generically computable.

Thus the nonuniform generic degrees of $X_0$ and $X_1$ form a minimal
pair, and hence so do their uniform generic degrees.

\medskip

To explain the basic idea for building $A_0$ and $A_1$, consider
requirements
\[
\mathcal P_{e,j} : |\dom \Phi_e \cap R_e|=\omega\; \Rightarrow\; \dom
\Phi_e \cap R_e \cap A_j \neq \emptyset
\]
and
\[
\mathcal N_{e_0,e_1} : \forall s\, \forall x\,
[x \in W_{e_0}^{f_{0,s}}[s] \cap
W_{e_1}^{f_{1,s}}[s]\;  \Rightarrow\;  \exists j\, [x
\in W_{e_j}^{f_j}]],
\]
arranged into a priority list as usual. (In fact, only the $\mathcal
P$-requirements need to be assigned priorities.)

Let us consider the interaction of a requirement $\mathcal
N_{e_0,e_1}$ with the requirements $\mathcal P_{e,j}$. (Different
$\mathcal N$-requirements will not interact with each other.) Whenever
we have $x \in W_{e_0}^{f_{0,s}}[s] \cap W_{e_1}^{f_{1,s}}[s]$, we
want to preserve at least one of the computations that have led to the
enumerations of $x$. Each $\mathcal P_{e,j}$ acts by waiting for a
witness $n \in \dom \Phi_e \cap R_e$ to appear, and then putting $n$
into $A_j$. Doing so might destroy computations we are trying to
preserve, but only on the $j$-side. We can now force all requirements
weaker than $\mathcal P_{e,j}$ to choose witnesses beyond the uses of
all computations we are trying to preserve, thus keeping the
$(1-j)$-side from harm by such requirements. (Notice that it is not
$\mathcal N_{e_0,e_1}$ that imposes this restraint, but $\mathcal
P_{e,j}$, which is why the restraint imposed on a particular $\mathcal
P$-requirement will be bounded, even as $\mathcal N_{e_0,e_1}$ has more
and more computations it wants to preserve.)

But what if a requirement $\mathcal P_{e',1-j}$ stronger than
$\mathcal P_{e,j}$ wants to act, say at stage $t>s$? Then we must let
it do so, which might destroy computations on the $(1-j)$-side. To
compensate for that possibility, we remove the numbers we have put
into $A_j$ for the sake of $\mathcal P_{e,j}$ (or any other
requirement weaker than $\mathcal P_{e',1-j}$). The key observation
here is the following: this action does not necessarily restore
$f_{j,t}$ to be the same as $f_{j,s}$, but assuming that $\mathcal
P_{e',1-j}$ is the strongest requirement to act since stage $s$, it
ensures that $f_{j,t} \supseteq f_{j,s}$, which means that every
element of $W_{e_j}^{f_{j,s}}[s]$ is also in $W_{e_j}^{f_{j,t}}[t]$.

\medskip

We now turn to the formal construction of $A_0$ and $A_1$. For each
$e,j$, we have a restraint $r(e,j)$, initially set to $0$. We adopt
the common convention that a use defined at stage $s$ cannot be larger
than $s$.

At stage $s$, let $\langle e,j \rangle<s$ be least such that $\dom
\Phi_e[s] \cap R_e \cap A_j[s] = \emptyset$ and there is an $n \in
\dom \Phi_e[s] \cap R_e$ that is larger than $r(e',j')$ for all
$\langle e',j' \rangle < \langle e,j \rangle$. (If there is no such
pair $e,j$ then proceed to the next stage.) We say that $\mathcal
P_{e,j}$ acts at stage $s$. Put $n$ into $A_j$. For each $\langle
e',1-j \rangle > \langle e,j \rangle$, remove every number put into
$A_{1-j}$ at a previous stage at which $\mathcal P_{e',1-j}$
acted. Let $r(e,j)=s$.

\medskip

We now verify that this construction has the desired properties.

If a requirement puts a number into $A_j$ and that number is later
removed at stage $s$, then any number put into $A_j$ by that same
requirement at a later stage will be bigger than $s$. Since only
$\mathcal P_{e,j}$ can ever put a number into $R_e \cap A_j$, and the
$R_e$'s are disjoint, each $A_j$ is $\Delta^0_2$ (in fact, d.c.e.).

Assume by induction that each $\mathcal P_{e',j'}$ with $\langle e',j'
\rangle < \langle e,j \rangle$ stops acting, say by stage $t$. Then
all the corresponding restraints $r(e',j')$ stop changing, so if $\dom
\Phi_e \cap R_e$ is infinite then either there is some $n \in \dom
\Phi_e[t] \cap R_e \cap A_j[t]$ or $\mathcal P_{e,j}$ eventually gets
to act after stage $t$ and put a number $n$ into $\dom \Phi_e \cap R_e
\cap A_j$. In either case, $n$ is never later removed from $A_j$, so
property (2) holds. Furthermore, $\mathcal P_{e,j}$ acts at most once
after stage $t$, so the induction can continue, and $\mathcal
P_{e,j}$ puts at most finitely elements into $A_j \cap R_e$ after
stage $t$ (in fact, at most one), so property (1) also holds.

Now fix $e_0,e_1$. Suppose that $x \in W_{e_0}^{f_{0,s}}[s] \cap
W_{e_1}^{f_{1,s}}[s]$. Let $\langle e,j \rangle$ be the least pair
such that $\mathcal P_{e,j}$ ever acts at a stage $t \geq s$. Then
every element put into $A_{1-j}$ between stages $s$ and $t$ is removed
from $A_{1-j}$ at stage $t$, so $f_{1-j,t} \supseteq f_{1-j,s}$, and
hence $x \in W_{e_{1-j}}^{f_{1-j,t}}[t]$. By the definition of
$r(e,j)$ at this stage, $f_{1-j}$ cannot change below the use
of this enumeration after stage $t$, so $x \in
W_{e_{1-j}}^{f_{1-j}}$. Thus property (3) is satisfied.
\end{proof}

The sets $X_j$ built in the proof of Theorem~\ref{main} have density
$1$, which means that they are coarsely computable, and is also
interesting in light of work of Igusa~\cite{Igusa2}: Astor,
Hirschfeldt, and Jockusch~\cite{AHJ} showed that the upper cone above
any nontrivial (nonuniform or uniform) generic degree has measure $0$,
so a minimal generic degree would necessarily be half of a minimal
pair. It is open whether there are minimal generic degrees, but
Igusa~\cite{Igusa2} showed that the generic degree of a density-$1$
set cannot be minimal. Interestingly, he also showed that if a uniform
generic degree does not bound a nontrivial uniform generic degree
containing a density-$1$ set, then it is half of a minimal pair, but
again it is not known whether such degrees exist.

It is also worth noting that the $X_j$ are $\Delta^0_2$. In fact,
they are both co-d.c.e. (Notice that a c.e.\ density-$1$ set is
generically computable.) The case of Theorem~\ref{igthm} where both
sets are $\Delta^0_2$ had been proved earlier by Downey, Jockusch, and
Schupp~\cite{DJS}. By Theorem~\ref{randthm}, the coarse degree of
every set that is not coarsely computable is half of a minimal pair,
and indeed forms minimal pairs with the coarse degrees of measure-$1$
many sets, but the following questions are open.

\begin{oq}
Are there $\Delta^0_2$ sets whose (nonuniform or uniform) coarse
degrees form a minimal pair?
\end{oq}

The same question can be asked for the notion of dense computability
introduced by Astor, Hirschfeldt, and Jockusch~\cite{AHJ}. (In that
paper, they showed that the resulting degree structures do have
minimal pairs.)

\begin{oq}
Is the (nonuniform or uniform) generic degree of every set that is not
generically computable half of a minimal pair? What is the measure of
the set of all $X \oplus Y$ such that the generic degrees of $X$ and
$Y$ form a minimal pair? More generally, what can be said about the
distribution of minimal pairs in the generic degrees?
\end{oq}

Astor, Hirschfeldt, and Jockusch~\cite{AHJ} also studied the following
notion, which had been briefly considered much earlier by
Meyer~\cite{Meyer} and Lynch~\cite{Lynch}.

\begin{defn}
An \emph{effective dense description} of a set $A$ is a (total)
function $f : \omega \to \{0,1,\square\}$ such that $f^{-1}(\square)$
has density $0$ and $f(n)=A(n)$ whenever $f(n) \in \{0,1\}$. A
set is \emph{effectively densely computable} if it has a
computable effective dense description.

We say that $A$ is \emph{nonuniformly effectively densely reducible}
to $B$ if every effective dense description of $A$ computes an
effective dense description of $B$.
We say that $A$ is \emph{uniformly effectively densely reducible} to
$B$ if there is a Turing functional $\Phi$ such that if $f$ is an
effective dense description of $B$, then $\Phi^f$ is an effective
dense description of $A$.
\end{defn}

These reducibilities lead to degree structures, for which the
following question remains open. (It is also open whether there are
minimal pairs for relative effective dense reducibility.)

\begin{oq}[Astor, Hirschfeldt, and Jockusch~\cite{AHJ}]
Are there minimal pairs in the (nonuniform or uniform) effective dense
degrees?
\end{oq}

It does not seem that the proof of Theorem~\ref{main} can be adapted
to this case in a straightforward way, because the key observation
mentioned in the informal description of the construction in that
proof does not apply here: Suppose that we make changes to the $j$-side
of a pair of convergent computations, then find another such pair, and
then again make changes to the $j$-side. If we later want to make
changes to the $(1-j)$-side, we can no longer restore both
computations on the $j$-side, because they might be based on different
oracles. Because we are dealing with Turing functionals rather than
enumeration operators, the only way to guarantee a return to a
previous computation is to return to the exact original oracle (up to
the relevant use). A proof along the lines of the usual construction
of a minimal pair of c.e.\ Turing degrees, where we try to preserve at
least one side of a pair of convergent computations up to their length
of agreement (see e.g.\ \cite[Section 2.14.2]{DowHir}), also seems
problematic, because if a computation converges on both sides to $0$,
say, and then the computation on one side disappears, when that
computation converges again, it might converge to $\square$, which now
allows the computation on the other side to change to $1$ without
creating a disagreement.

\end{document}